\documentclass[12pt]{article}

\usepackage{amsmath,amsthm,amssymb,amscd,easywncy,mathdots,enumerate}
\usepackage{stmaryrd}
\usepackage{setspace}
\usepackage{tikz-cd}
\usepackage{tikz}
\usepackage{float}
\usetikzlibrary{calc}
\usepackage[belowskip=-20pt]{caption}

\allowdisplaybreaks

\newtheorem{theorem}{Theorem}[section]
\newtheorem{proposition}[theorem]{Proposition}

\newtheorem{definition}[theorem]{Definition}
\newtheorem{lemma}[theorem]{Lemma}

\newtheorem{conjecture}[theorem]{Conjecture}
\newtheorem{ex}[theorem]{Example}
\newtheorem{remark}[theorem]{Remark}
\newtheorem{question}[theorem]{Question}

\numberwithin{equation}{section}

\newcommand{\h}{\mathfrak H}
\newcommand{\g}{\mathfrak H^2}
\newcommand{\G}{\mathcal G}

\newcommand{\R}{\mathbb R}
\newcommand{\mz}{\mathcal Z}

\newcommand{\Z}{{\mathbb Z}}
\newcommand{\N}{{\mathbb N}}
\newcommand{\C}{{\mathbb C}}
\newcommand{\Q}{{\mathbb Q}}

\newcommand{\gsh}{\boxdot}
\newcommand{\sh}{\mathcyr {sh}}
\newcommand{\shh}{\,\mathcyr {sh}\,}
\newcommand{\mg}{\mathfrak{g}}

\newcommand{\dif}{ \operatorname{d}}

\DeclareMathOperator{\ds}{\operatorname{ds}}
\newcommand{\mes}{\mathcal{E}}
\newcommand{\geng}{T}
\newcommand{\bigslant}[2]{{\raisebox{.2em}{$#1$}\left/\raisebox{-.2em}{$#2$}\right.}}

\DeclareRobustCommand{\mb}{\genfrac{[}{]}{0pt}{}}

\DeclareRobustCommand{\HH}{H\genfrac{(}{)}{0pt}{}}
\DeclareRobustCommand{\MM}{M\genfrac{(}{)}{0pt}{}}
\DeclareRobustCommand{\GG}{\geng\genfrac{(}{)}{0pt}{}}

\title{Double shuffle relations for $q$-analogues of multiple zeta values, their derivatives and the connection to multiple Eisenstein series}
\author{Henrik Bachmann}
\begin{document}
\date{\today}
\maketitle

\begin{abstract} We study a certain class of $q$-analogues of multiple zeta values, which appear in the Fourier expansion of multiple Eisenstein series. Studying their algebraic structure and their derivatives we propose conjectured explicit formulas for the derivatives of double and triple Eisenstein series.
\end{abstract}

\section{Introduction}
For $k_1,\dots,k_{r-1} \geq 1, k_r \geq 2$ the multiple zeta value $\zeta(k_1,\ldots,k_r)$ is defined by
\begin{equation} \label{eq1_2}
\zeta(k_1,\ldots,k_r)=\sum_{0<m_1<\cdots<m_r} \frac{1}{m_1^{k_1}\cdots m_r^{k_r}} .
\end{equation}
By $r$ we denote its depth, $k_1+\dots+k_r$ will be called its weight and for the $\Q$-vector space spanned by all multiple zeta values we write $\mz$. These numbers have been studied recently in many different contexts in mathematics and theoretical physics. In \cite{gkz} the authors studied several connections of double zeta values (the $r=2$ case of \eqref{eq1_2}) to modular forms for the full modular group. One famous result of \cite{gkz} is the relationship between linear relations between $\zeta(a,b)$ with both $a$ and $b$ beeing odd and cusp forms of weight $a+b$. For example it was shown, that the coefficient of the period polynomial of the first non trivial cusp form $\Delta$ in weight $12$ can be used to obtain the relation 
\begin{equation}\label{eq:exotic}
\frac{5197}{691} \zeta(12) =  168 \zeta(7,5)+150 \zeta(5,7) + 28 \zeta(3,9) \,.
\end{equation}
Further it is conjectured, that there is a one-to-one correspondence between cusp forms and these type of relations among double zeta values. Another connection between double zeta values and modular form which was first introduced in \cite{gkz} are double Eisenstein series. These can be seen as a mixture of classical Eisenstein series and double zeta values. The higher depth case, the multiple Eisenstein series, where then studied in \cite{Ba}. For $k_1,\dots,k_r \geq 2$ the multiple Eisenstein series $G_{k_1,\dots,k_r}(\tau)$ is defined\footnote{Since the sum in \eqref{eq_gk} is just absolute convergent in the cases $k_r\geq 3$ one uses Eisenstein summation for $k_r =2$} by 
\begin{equation} \label{eq_gk}
 G_{k_1,\ldots,k_r}(\tau) =   \sum_{\substack{0\prec \lambda_1\prec \cdots\prec \lambda_r\\ \lambda_i\in\Z \tau+\Z}} \frac{1}{\lambda_1^{k_1}\cdots \lambda_r^{k_r}} \,,
 \end{equation}
where $\tau \in \left\{ x+iy \in \C \mid  y>0 \right\}$ is an element in the upper half plane and the order $\prec$ on $\Z \tau + \Z$ is defined by $m_1 \tau + n_1 \prec m_2 \tau + n_2 :\Leftrightarrow (m_1 < m_2) \vee (m_1 = m_2 \wedge n_1 < n_2 )$. In the case $r=1$ these are the classical Eisenstein series which have the following Fourier expansion ($k\geq 2$)
	\[ G_k(\tau) = \zeta(k) + \frac{(-2\pi i)^k}{(k-1)!} \sum_{n = 1}^\infty \sigma_{k-1}(n) q^{n}\,\qquad (q=e^{2\pi i \tau}) \,,\]
with the divisor-sum $\sigma_{k-1}(n)=\sum_{d | n} d^{k-1} $. The main result of \cite{Ba} was that the multiple Eisenstein series also have a Fourier expansion and that it can be written as
\[  G_{k_1,\ldots,k_r}(\tau) =  \zeta(k_1,\dots,k_r) + \sum_{\substack{1<l<r\\m_1+\dots+m_r=k_1+\dots+k_r}} \alpha_{m_1,\dots,m_l} \cdot \hat{g}_{m_{l+1},\dots,m_r}(q) +\hat{g}_{k_1,\dots,k_r}(q) \,, \]
where the $\alpha_{m_1,\dots,m_l} $ are  $\Q$-linear combinations of multiple zeta values of depth $l$ and weight $m_1+\dots+m_l$ and $\hat{g}_{k_1,\dots,k_r}(q) = (-2 \pi i)^{k_1+\dots+k_r} g_{k_1,\dots,k_r}(q)$. The series $ g_{k_1,\dots,k_r}(q) \in \Q[[q]]$ will be studied in detail in this work and its coefficient can be seen as a multiple version of the divisor sums. 

 By some classical results of modular forms together with the results in \cite{Ba} or \cite{BT} it can be shown that every modular form can be written in terms of multiple Eisenstein series. For example it is
\[
\frac{(2\pi i)^{12}}{2^6 \cdot 5 \cdot 691} \cdot \Delta = \frac{5197}{691} G_{12} - 168 G_{7,5}-150 G_{5,7} - 28 G_{3,9} \,,
\]
which gives another way to prove the relation \eqref{eq:exotic} since the constant term of the Fourier expansion of the cusp form on the left hand side vanishes.

Since there just exist multiple Eisenstein series for the cases $k_1,\dots,k_r \geq 2$ a natural question was if there is an extended definition of  $G_{k_1,\ldots,k_r}(\tau)$ for the cases in which the multiple zeta value $\zeta(k_1,\ldots,k_r)$ exists. This question was answered in \cite{BT}, where the authors introduced the functions  $G^\sh_{k_1,\ldots,k_r}(\tau)$ for all $k_1,\dots,k_r \geq 1$, which coincide with  $G_{k_1,\ldots,k_r}(\tau)$ in the cases $k_1,\dots,k_r \geq 2$. These series have a Fourier expansion of the form 
\[  G^\sh_{k_1,\ldots,k_r}(\tau) =  \zeta^\sh(k_1,\dots,k_r) + \sum_{\substack{1<l<r\\m_1+\dots+m_r=k_1+\dots+k_r}} \alpha_{m_1,\dots,m_l} \cdot \hat{g}^\sh_{m_{l+1},\dots,m_r}(q) +\hat{g}^\sh_{k_1,\dots,k_r}(q) \,, \]
where the $\zeta^\sh(k_1,\dots,k_r) \in \mz $ are the shuffle-regularized multiple zeta values (\cite{IKZ}) and again $ \alpha_{m_1,\dots,m_l} \in \mz$. Here it is $\hat{g}^\sh_{k_1,\dots,k_r}(q) = (-2 \pi i)^{k_1+\dots+k_r} g^\sh_{k_1,\dots,k_r}(q)$, where the $g^\sh$ can be seen as "shuffle regularized" versions of the functions $g$. For example it is
\begin{align*}
G^\sh_{1,3}(\tau) = \zeta(1,3) - \zeta(2)\cdot (2 \pi i)^2\cdot g^\sh_2(q) + (2\pi i)^4\cdot g^\sh_{1,3}(q)\,. 
\end{align*}
We will study the algebraic structure of the series $g^\sh_{k_1,\dots,k_r}(q)$, to make progress towards a question on multiple Eisenstein series and their derivatives which we will describe in the following.

Denote by $\mathcal{E}$ the $\Q$-vector space spanned by all $G^\sh_{k_1,\ldots,k_r}$ for\footnote{We set $G^\sh_{k_1,\ldots,k_r}(\tau)=1$ for $r=0$.} $r\geq 0$ and $k_1,\dots,k_r \geq 1$ and consider the projection $\pi$ to the constant term in the Fourier expansion, i.e.
\begin{align*}
	\pi : \mathcal{E} &\longrightarrow \mz \\
	  G^\sh_{k_1,\ldots,k_r} &\longmapsto \zeta^\sh(k_1,\dots,k_r) \,.
\end{align*} 
Since the space of modular forms is contained in the space $\mathcal{E}$ it is clear that the space of cusp forms is contained in the kernel of the map $\pi$. 

It is therefore an interesting question if the kernel of $\pi$ consists more than just of cusp forms. In fact there are already non-trivial elements in the kernel of $\pi$ in weight $3$. Since $\zeta(1,2) - \zeta(3)=0$ it is $G^\sh_{1,2}-G^\sh_3 \in \ker \pi$, but $G^\sh_{1,2}-G^\sh_3  \neq 0$. We will see that $G^\sh_{1,2} - G^\sh_3 = \frac{(2\pi i)^2 }{2} \dif G_1$, where the operator $\dif=q \frac{d}{dq}$ plays also an important role in the theory of modular forms. Another way of interpreting this is that $(2\pi i)^2 \dif G_1$ is again an element in $\mes$.  In general it is not known, but expected, that the space $\mes$ is closed under $\dif$, i.e. $(2\pi i)^2\dif \mes \subset \mes$. This question will be one motivation for the present paper.

For this we will study two types of $q$-series. The first one, first introduced in \cite{BK} and \cite{B}, will be the double indexed series $g^{(d_1,\dots,d_r)}_{k_1,\dots,k_r}(q)$ for $d_1,\dots,d_r \geq 0$ and $k_1,\dots,k_r\geq 1$. Similar to multiple zeta values there are two different ways to express the product of these series and we will describe this double shuffle structure in detail. The other series, already mentioned before, are the $g^\sh_{k_1,\dots,k_r}(q)$ appearing in the Fourier expansion of multiple Eisenstein series. The $g^\sh$ can be written explicitly in terms of the double-indexed $g$. 
Though the behavior of $g$ under the operator $\dif$ is well-understood (See Section \ref{sec:difg}), the behavior of $g^\sh$ under this operator is an open problem.
\begin{figure}[H]	

\begin{center}
	\includegraphics{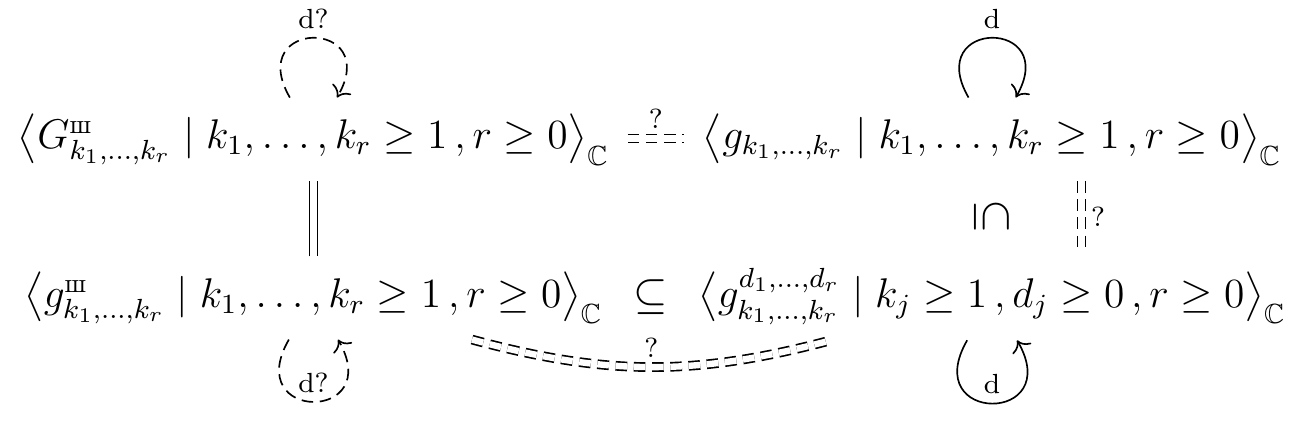}
	\caption{Overview of the spaces spanned by $G^\sh, g$ and $g^\sh$ and the behavior of the operator $\dif = q \frac{d}{dq}$. The dashed equalities and lines are expected but unproven so far.}
\end{center}
\end{figure}

Since every $G^\sh$ can be written as a $\C$-linear combination of $g^\sh$ and vice versa, the space spanned by them are the same. Therefore to prove that the multiple Eisenstein series are closed under the operator $\dif$ it suffices to prove it for the functions $g^\sh$. We will present new results on this and prove the following:
\begin{theorem}\label{thm:intro} \begin{enumerate}[i)]
		\item  For $k\geq 1$ and $\dif = q \frac{d}{dq}$ we have
		\begin{equation*} 
		\frac{1}{k} \dif g^\sh_k = (k+1) g^\sh_{k+2} - \sum_{n=2}^{k+1} (2^n - 2) g^\sh_{k+2-n,n} \,.
		\end{equation*}		
		\item For $k_1,k_2,k_3 \geq 2$ the series $\dif g^\sh_{k_1,k_2}$ and $\dif g^\sh_{k_1,k_2,k_3}$ can be written as linear combinations of $g^\sh$.
	\end{enumerate}
\end{theorem}

For Theorem \ref{thm:intro} ii) we will present explicit formulas for the mentioned linear combination modulo lower weight terms (Theorem \ref{thm:dgshlen2len3}). Since it is expected (Question \ref{qu:mes}) that the space spanned by the $g^\sh$ modulo lower weight terms has the same algebraic structure as the space of multiple Eisenstein series this will lead us to propose conjectures on explicit formulas for the derivative of double- and triple Eisenstein series.\\

{\bf Acknowledgments}\\
The author would like to thank Hidekazu Furusho, Ulf K\"uhn, Koji Tasaka and Wadim Zudilin for their helpful comments and suggestions. This work was written while the author was an Overseas researcher under Postdoctoral Fellowship of Japan Society for the Promotion of Science. It was partially supported by JSPS KAKENHI Grant Number JP16F16021.

\section{Algebraic setup}

We recall Hoffman`s algebraic setup for quasi-shuffle products (\cite{H1}) but with slightly different notations. First we start with the two product structures coming from the theory of multiple zeta values. After this we introduce an analogue setup for the product structure of the $q$-analogues which will be introduced in the next section.
\subsection{Classical case}
Write $\h = \Q\langle e_0,e_1\rangle$ for the non commutative polynomial algebra of indeterminates $e_0$ and $e_1$ over $\Q$, and define its subalgebras $\h^0$ and $\h^1$ by
\[ \h^0 = \Q\cdot 1 + e_1 \h e_0 \subset \h^1 = \Q \cdot 1 + e_1 \h \subset \h \,,\]
where $1$ denotes the empty word here. For $k\geq 1$ we put $e_k = e_1 e_0^{k-1}$, so that the monomials $e_{k_1} \dots e_{k_r}$ form a basis of $\h^1$ and the monomials $e_{k_1} \dots e_{k_r}$ with $k_r \geq 2$ form a basis of $\h^0$.

We consider two $\Q$-bilinear commutative products $\sh$ on $\h$ and $\ast$ on $\h^1$, called the shuffle and the harmonic (or stuffle) products, which are defined by
\begin{align*}
1 \shh w = w \shh 1 &= w \qquad (w \in \h)\,,\\
av \shh bw = a (v \shh bw) &+ b(av \shh w) \qquad (a,b, \in \{ e_0, e_1\}, v,w \in \h)
\end{align*}
and 
\begin{align*}
1 \ast w = w \ast 1 &= w \qquad (w \in \h^1)\,,\\
e_{k_1} v \ast e_{k_2} w = e_{k_1} (v \ast e_{k_2} w) &+ e_{k_2} (e_{k_1} v \ast w) + e_{k_1+k_2} (v \ast w) \qquad (k_1,k_2 \geq 1, v,w \in \h^1)\,.
\end{align*}
Denote by $\h_\sh$ (resp. $\h^1_\ast$) the commutative $\Q$-algebra $\h$ (resp. $\h^1$) equipped with the multiplication $\sh$ (resp. $\ast$). It is easy to see that the subspaces $\h^1$ and $\h^0$ of $\h$ (resp. the subspace $\h^0$ of $\h^1$) are closed under $\sh$ (resp. $\ast$) and we therefore write $\h_\sh^1$ and $\h_\sh^0$ (resp. $\h^0_\ast$) for the corresponding subalgebras of $\h_\sh$ (resp. $\h^1_\ast$).

Identifying an indexset $(k_1,\dots,k_r)$ with the word $e_{k_1}, \dots, e_{k_r}$ it is easy to see that the indexsets for which the multiple zeta values $\zeta(k_1,\dots,k_r)$ exists correspond exactly to the words in $\h^0$. One therefore can interpret the multiple zeta values as a $\Q$-linear map $\zeta : \h^0 \rightarrow \R$, where we send the empty word in $\h^0$ to $1$. It is well known that this map is a $\Q$-algebra homomorphism from both $\h^0_\sh$ and $\h^0_\ast$ to $\R$, i.e. in particular it is
\begin{equation}\label{eq:dshmzv}
    \zeta(w \shh v) = \zeta(w) \cdot \zeta(v) =  \zeta(w \ast v) \,,
\end{equation}
for any words $w,v \in \h^0$. These relations are known as (finite) double shuffle relations. Another well known fact (See \cite{IKZ}) is, that the map $\zeta$ can be extended to algebra homomorphisms $\zeta^\sh :  \h^1_\sh \rightarrow \R$ and  $\zeta^\ast:  \h^1_\ast \rightarrow \R$, which are uniquely determined by $\zeta^\sh(e_1) = \zeta^\ast(e_1) =0$ and $\zeta^\sh(w) = \zeta^\ast(w) = \zeta(w)$ for $w\in \h^0$. 

Define for words $u,v \in \h^1$ the element $\ds(u,v) \in \h^1$ by
\begin{equation}\label{eq:defds}\ds(u,v) = u \ast v - u \shh v \,. 
\end{equation}
If both $u,v \in \h^0$ we have by \eqref{eq:dshmzv} that $\zeta(\ds(u,v))=0$. But more generally we have the following Theorem, which conjecturally gives all linear relations between multiple zeta values.

\begin{theorem} (Extended double shuffle relations)\label{thm:edsh} For $u \in \h^0$ and $v \in \h^1$ it is
\[ \zeta^\sh(\ds(u,v)) = \zeta^\ast(\ds(u,v)) = 0 \,.\]
\end{theorem}
\begin{proof} This is Theorem 1  together with Theorem 2 (iv) and (iv`)  in \cite{IKZ}.
\end{proof}

\subsection{Setup for the $q$-analogue case}
We now want to recall a similar algebraic setup from \cite{B} for our $q$-analogue which will be defined in the next section. In analogue to the space $\h^1$, which was spanned by words in the letters $e_{k}$ with $k\geq 1$,  we will now consider the space $\g$ spanned by words in the double-indexed letters $e^{(d)}_k$ with $k\geq 1$ and $d\geq 0$. More precisely we define  $\g = \Q\langle A \rangle$ to be the noncommutative polynomial algebra of indeterminates $A = \{ e^{(d)}_k \mid  k\geq 1\,, d\geq 0 \}$  over $\Q$. 
\begin{definition}("Harmonic product analog" $\boxast$ on $\g$)
In analogy to the product $\ast$ on $\h^1$ we define the product $\boxast$ on $\g$ by $1 \boxast w = w \boxast 1 = w$ for $w \in \g$ and 
\begin{align*}
e^{(d_1)}_{k_1} v \boxast e^{(d_2)}_{k_2} w &= e^{(d_1)}_{k_1} (v \boxast e^{(d_2)}_{k_2} w) + e^{(d_2)}_{k_2} (e^{(d_1)}_{k_1} v \boxast w) + \binom{d_1+d_2}{d_1} e^{(d_1+d_2)}_{k_1+k_2} (v \boxast w)  \\
&+\left(\binom{d_1+d_2}{d_1}\sum_{j=1}^{k_1} \lambda^j_{k_1,k_2}  e^{(d_1+d_2)}_{j} +\binom{d_1+d_2}{d_1}\sum_{j=1}^{k_2} \lambda^j_{k_2,k_1}  e^{(d_1+d_2)}_{j} \right) (v \boxast w) \,,
\end{align*}
where the numbers $\lambda^j_{a,b}  \in \Q$ for $1 \leq j \leq a$ are defined as
\[ \lambda^j_{a,b} := (-1)^{b-1} \binom{a+b-j-1}{a-j} \frac{B_{a+b-j}}{(a+b-j)!} \,, \]
with $B_k$ being the $k$-th Bernoulli number.	
\end{definition}
 It can be checked that $\g$ equipped with this product becomes a commutative $\Q$-algebra which be denote by $\g_\boxast$ (\cite{B}, Theorem 3.6.). For example we have
\begin{align}\label{eq:gst23}
e^{(0)}_{2} \boxast e^{(0)}_{3}& = e^{(0)}_{2}e^{(0)}_{3} + e^{(0)}_{3}e^{(0)}_{2}+e^{(0)}_{5} - \frac{1}{12} e^{(0)}_{3} \,,\\\label{eq:gst1121}
e^{(1)}_{1} \boxast e^{(2)}_{1}&= e^{(1)}_{1}e^{(2)}_{1}+e^{(2)}_{1}e^{(1)}_{1}+3e^{(3)}_{2}-3e^{(3)}_{1}\,.
\end{align}
Notice that up to the term $- \frac{1}{12} e^{(0)}_{3}$ equation \eqref{eq:gst23} looks exactly like the harmonic product $e_{2} \ast e_{3} = e_{2}e_{3} + e_{3}e_{2}+e_{5} $ in $\h^1$.

 The reason for introducing double-indexes, i.e. the $d_j$, will become clear now when we will introduce the second product on $\g$ corresponding to the shuffle product $\sh$ on $\h^1$.
For this we first define for a fixed $r$ the following generating series of monomials in depth $r$
\[ \MM{X_1,\dots,X_r}{Y_1,\dots,Y_r} := \sum_{\substack{k_1,\dots,k_r \geq 1 \\ d_1, \dots, d_r \geq 0}} e^{(d_1)}_{k_1} \dots e^{(d_r)}_{k_r} X_1^{k_1-1} \dots X_r^{k_r-1} \cdot Y_1^{d_1} \dots Y_r^{d_r} \]
as an element in $\g[[X_1,\dots,X_r,Y_1,\dots,Y_r]]$, i.e. the variables $X_i$ and $Y_j$ are commuting for $1\leq i,j \leq r$.  

\begin{definition}
For $k_1,\dots,k_r \geq 1$, $d_1,\dots,d_r \geq 0$ and $w=e^{(d_1)}_{k_1}, \dots, e^{(d_r)}_{k_r}$ define $P(w)$ as the coefficients of
\[ \sum_{\substack{k_1,\dots,k_r \geq 1 \\ d_1, \dots, d_r \geq 0}} P(w) X_1^{k_1-1} \dots X_r^{k_r-1} \cdot Y_1^{d_1} \dots Y_r^{d_r} :=  \MM{Y_r,Y_{r-1}+Y_r,\dots,Y_1+\dots+Y_r}{X_r-X_{r-1}, X_{r-1}-X_{r-2},\dots,X_1} \,. \]
We define the $\Q$-linear map $P: \g \rightarrow \g$ by setting $P(1)=1$ and extending the above definition on monomials linearly to $\g$.
\end{definition}
Notice that the map $P$ is an involution on $\g$, i.e. $P(P(w))=w$ for all $w\in \g$. For $r=1$ the definition reads
\[ \sum_{\substack{k_1 \geq 1 \\ d_1\geq 0}} P(e^{(d_1)}_{k_1}) X_1^{k_1-1} Y_1^{d_1} :=  \MM{Y_1}{X_1} = \sum_{\substack{k_1 \geq 1 \\ d_1\geq 0}} e^{(d_1)}_{k_1} Y_1^{k_1-1} X_1^{d_1}   \]
and therefore $ P(e^{(d_1)}_{k_1})= e^{(k_1-1)}_{d_1+1}$. Other examples are
\begin{align}\label{eq:p1}
P(e^{(2)}_{1} e^{(1)}_{1}) &= e^{(0)}_{2}e^{(0)}_{3}+3e^{(0)}_{1}e^{(0)}_{4}    \,,\\\label{eq:p2}
P(e^{(1)}_{1} e^{(2)}_{1}) &=e^{(0)}_{3}e^{(0)}_{2}+ 2e^{(0)}_{2}e^{(0)}_{3}+3e^{(0)}_{1}e^{(0)}_{4}    
\end{align}
which can be obtained by calculation the coefficient of $X_1^0 X_2^0 Y_1^2 Y_2^1$ (resp. $X_1^0 X_2^0 Y_1^1 Y_2^2$) in $\MM{Y_2,Y_1+Y_2}{X_2-X_1,X_1}$.

\begin{definition}("Shuffle product analog" $\gsh$ on $\g$)
Define on $\g$ the product $\gsh$ for $u,v \in \g$ by
\[ u \gsh v = P( P(u) \boxast P(v) ) \,.  \]
\end{definition}
This product is commutative and associative which follows from the fact that $P$ is an involution together with the properties of $\boxast$. We denote by $\g_\gsh$ the corresponding $\Q$-algebra equipped with this product. 

\begin{ex} \label{ex:shuffleg}
We now use \eqref{eq:gst1121}, \eqref{eq:p1} and  \eqref{eq:p1} together with $P(e^{(0)}_{2})= e^{(1)}_{1}$ and $P(e^{(0)}_{3})= e^{(2)}_{1}$ to calculate $ e^{(0)}_{2} \gsh  e^{(0)}_{3}$:
\begin{align*}
e^{(0)}_{2} \gsh  e^{(0)}_{3} &=\,\, P( P(e^{(0)}_{2}) \boxast P(e^{(0)}_{3}) ) = P(e^{(1)}_{1} \boxast  e^{(2)}_{1}) \\
&=\,\, P(e^{(1)}_{1}e^{(2)}_{1}+e^{(2)}_{1}e^{(1)}_{1}+3e^{(3)}_{2}-3e^{(3)}_{1})\\
&=\,\,e^{(0)}_{3}e^{(0)}_{2}+ 3e^{(0)}_{2}e^{(0)}_{3}+6e^{(0)}_{1}e^{(0)}_{4}  +3e^{(1)}_{4}-3e^{(0)}_{4} \,.
\end{align*}
Compare this to the shuffle product $e_{2} \shh  e_{3} = e_{3}e_{2}+ 3e_{2}e_{3}+6e_{1}e_{4}$ on $\h^1_\sh$.
\end{ex}

\section{Certain $q$-analogues of multiple zeta values}
In recent years several different $q$-analogues of multiple zeta values have been studied. An overview of these different models can be for example found in \cite{Zh}. Our model we present here has its motivation in its appearance in the Fourier expansion of multiple Eisenstein series. It was first studied in \cite{BK} and later in more detail in \cite{B}.
In this section we want to introduce two types of $q$-series which are closely related to each other. We will construct two maps, where the first one, denoted by $\mg$, will be an algebra homomorphism from both $\g_\boxast$ and $\g_\gsh$ to $\Q[[q]]$. The multiplication of $\Q[[q]]$ here is the usual multiplication of formal $q$-series. Similar to the case of multiple zeta values we will obtain a large family of linear relations between these $q$-series, by comparing $\mg(u \boxast v)$ and $\mg(u \gsh v)$.

The second map, denoted by $\mg^\sh$, will be more closely related to multiple zeta values, since it will be an algebra homomorphism from $\h_\sh^1$ to $\Q[[q]]$. 

\subsection{The series $g^{(d_1,\dots,d_r)}_{k_1,\dots,k_r}$ and the map $\mg$}
In this section we will recall some of the result of \cite{B} and \cite{BK}. Here we use a different notation which matches the one used in \cite{BT}.

\begin{definition}
	For $k_1,\dots,k_r \geq 1$, $d_1,\dots,d_r \geq 0$  we define the following $q$-series\footnote{In \cite{B} a different notation and order was used. There the series $	g^{(d_1,\dots,d_r)}_{k_1,\dots,k_r}$ was called bi-bracket and it was denoted by $\mb{k_r,\dots,k_1}{d_r,\dots,d_1}$ and instead of $\G$ the author used $\mathcal{BD}$. }
	\begin{align*}
	g^{(d_1,\dots,d_r)}_{k_1,\dots,k_r}(q) :=& \sum_{\substack{0 < u_1 < \dots < u_r \\0<v_1, \dots , v_r }} \frac{u_1^{d_1}}{d_1!} \dots \frac{u_r^{d_r}}{d_r!} \cdot \frac{v_1^{k_1-1} \dots v_r^{k_r-1}}{(k_1-1)!\dots(k_r-1)!}   \cdot q^{u_1 v_1 + \dots + u_r v_r} \in \Q[[q]]\,.
	\end{align*}
	By $k_1+\dots+k_r + d_1+\dots+d_r$ we denote its weight and by $r$ its depth. 
\end{definition}

 Since $q$ will be fixed the whole time we will also write $g^{(d_1,\dots,d_r)}_{k_1,\dots,k_r}$ instead of $g^{(d_1,\dots,d_r)}_{k_1,\dots,k_r}(q)$. For the $\Q$-vector space spanned by all of these $q$-series we write

\[ \G := \big\langle g^{(d_1,\dots,d_r)}_{k_1,\dots,k_r}  \mid  r\ge0, k_1,\dots,k_r \ge 1\,, d_1,\dots,d_r \ge 0 \big\rangle_{\Q} \,, \]
 where we set $g^{(d_1,\dots,d_r)}_{k_1,\dots,k_r}=1$ for $r=0$. In the case $d_1=\dots=d_r=0$ we write\footnote{The series $g_{k_1,\dots,k_r}$ were first studied in \cite{BK}, where the author referred to it as brackets and denoted it by $[k_r,\dots,k_1]$. The space $\G^{(0)}$ was denoted $\mathcal{MD}$ there.} 
 \[g_{k_1,\dots,k_r} := g^{(0,\dots,0)}_{k_1,\dots,k_r} \] 
 and denote the subspace spanned by all of these by
\[ \G^{(0)} := \big\langle g_{k_1,\dots,k_r}  \mid  r\ge0, k_1,\dots,k_r \ge 1 \big\rangle_{\Q} \subset  \G \,. \]

\begin{definition}
We define the $\Q$-linear map $\mg$ from $\g$ to $\G$ on the monomials by
\begin{align*}
\mg : \g &\longrightarrow \G \,,\\
w = e^{(d_1)}_{k_1} \dots e^{(d_r)}_{k_r} &\longmapsto \mg(w) := g^{(d_1,\dots,d_r)}_{k_1,\dots,k_r} \,,
\end{align*}
and set $\mg(1) = 1$.
\end{definition}

\begin{theorem} \label{thm:gdsh} The following statements hold for the map $\mg$.
	\begin{enumerate}[i)]
		\item The map $\mg$ is invariant under $P$, i.e.  for all $w\in \g$ it is
		\[ \mg(P(w))=\mg(w)  \,. \]
		\item $\mg$ is an algebra homomorphism from $\g$ to $\Q[[q]]$ with respect to both products $\boxast$ and $\gsh$, i.e. we have for all $u,v \in \g$
		\[  \mg( u \gsh v) = \mg(u) \cdot \mg(v) = \mg( u \boxast v) \,, \]
		where $\cdot$ denotes the usual multiplication of formal $q$-series in $\Q[[q]]$. In particular the space $\G = \mg(\g) \subset \Q[[q]]$ is an $\Q$-algebra.
	\end{enumerate}
\end{theorem}
\begin{proof}
	The first statement is Theorem 2.3 (Partition relation) in \cite{B}. It has a nice description using the conjugation of partitions, which is the reason for the name of the map $P$. The second statement is Theorem 3.6. in \cite{B}. 
\end{proof}

The statement ii) in Theorem \ref{thm:gdsh} can be seen as double shuffle relations for the $q$-series 	$g^{(d_1,\dots,d_r)}_{k_1,\dots,k_r}$ similar to the double shuffle relations \eqref{eq:dshmzv} of multiple zeta values.
\begin{ex} 
	We have seen before that
	\begin{align*}
		e^{(0)}_{2} \boxast  e^{(0)}_{3} &= e^{(0)}_{2}e^{(0)}_{3} + e^{(0)}_{3}e^{(0)}_{2}+e^{(0)}_{5} - \frac{1}{12} e^{(0)}_{3} \,,\\
	e^{(0)}_{2} \gsh  e^{(0)}_{3} &=e^{(0)}_{3}e^{(0)}_{2}+ 3e^{(0)}_{2}e^{(0)}_{3}+6e^{(0)}_{1}e^{(0)}_{4}  +3e^{(1)}_{4}-3e^{(0)}_{4} 
	\end{align*}  
	and therefore we obtain the relation 
	\begin{align*}
	0 &= \mg(e^{(0)}_{2} \boxast  e^{(0)}_{3}) -\mg(e^{(0)}_{2} \gsh  e^{(0)}_{3}) =g_5 - 2 g_{2,3} - 6 g_{1,4} - 3g^{(1)}_4 + 3g_4 - \frac{1}{12}g_3 \,.
	\end{align*}
\end{ex}

Since $\h^1$ and $\h^0$ have a natural embedding in $\g$ by sending a monomial $e_{k_1} \dots e_{k_r}$ to $e^{(0)}_{k_1} \dots e^{(0)}_{k_r}$ we will view both $\h^1$ and $\h^0$ as subspaces of $\g$ in the following, i.e.
\[\h^0 \subset \h^1 \subset \g \,.\]

 In particular we can view $\mg$ as a map from $\h^1$ (resp. $\h^0$) to $\G$. Clearly the image of $\h^1$ under $\mg$ is exactly the space $\G^{(0)}=\mg(\h^1)$.

\begin{proposition}\label{prop:h1stuffle} The spaces $\h^1$ and $\h^0$ are closed under $\boxast$ and therefore we also have for $u,v \in \h^1$ (resp. $\h^0$) that
			\[ \mg(u) \cdot \mg(v) = \mg( u \boxast v) \,. \]
In particular the space $\G^{(0)}$ is a subalgebra of $\G$.
\end{proposition}
\begin{proof} This follows directly from the definition of the product $\boxast$, since it does not increase the indexes $d_j$.
\end{proof}

Notice that the analogue statement of Proposition \ref{prop:h1stuffle} for the product $\gsh$ is false, since by Example \ref{ex:shuffleg} we have  $e_2 \gsh e_3 \notin \h^1$.

\begin{remark} Even though it is not the purpose of this paper we give a remark on why the series $g$ can be considered as a $q$-analogue of multiple zeta values. This was discussed in $\cite{BK}$, where the authors introduced the following map. Define for $k\in \N$ the map $\Q[[q]] \rightarrow \R \cup \{ \infty \}$ by $Z_k(f) = \lim_{q \to 1} (1-q)^{k}f(q)$. One can show (\cite{BK}, Proposition 6.4.) that for  $k_1,\dots,k_{r-1} \geq 1, k_r \geq 2$ and $k=k_1+\dots+k_r$ it is
\[ Z_k(g_{k_1,\dots,k_r}) = \zeta(k_1,\dots,k_r) \,. \]	
In this note we will not focus on this aspect in more detail.
\end{remark}

We end this section by discussing the generating series of our $q$-series $g^{(d_1,\dots,d_r)}_{k_1,\dots,k_r}$, since we will need them in the remaining sections. By Theorem  2.3 in \cite{B} we have the following explicit expression
	\begin{align}\label{eq:genseriesg}\begin{split}
\GG{X_1,\dots,X_r}{Y_1,\dots,Y_r} &:= \sum_{\substack{k_1,\dots,k_r \geq 1\\d_1,\dots,d_r \geq 0}} g^{(d_1,\dots,d_r)}_{k_1,\dots,k_r} X_1^{k_1-1} \dots X_r^{k_r-1}  Y_1^{d_1} \dots Y_r^{d_r}  \\
&= \sum_{0 < n_1 < \dots < n_r} e^{n_1 Y_1}\frac{e^{X_1} q^{n_1}}{1-e^{X_1} q^{n_1}} \dots  e^{n_r Y_r} \frac{e^{X_r} q^{n_r}}{1-e^{X_r} q^{n_r}} \,. 
	\end{split}
	\end{align}
		
		Notice that with this the invariance of the map $\mg$ under the involution $P$ (Theorem \ref{thm:gdsh} i) ) can be stated as 
\begin{equation}\label{eq:partition}
\GG{X_1,\dots,X_r}{Y_1,\dots,Y_r} = \GG{Y_r,Y_{r-1}+Y_r,\dots,Y_1+\dots+Y_r}{X_r-X_{r-1}, X_{r-1}-X_{r-2},\dots,X_1} \,.
\end{equation}	
		
For the generating series of the $q$-series $g_{k_1,\dots,k_r} = g^{(0,\dots,0)}_{k_1,\dots,k_r}$ we will write
	\begin{align}\label{eq:genserg2}\begin{split}
	\geng(X_1,\dots,X_r) &:= \GG{X_1,\dots,X_r}{0,\dots,0}=\sum_{k_1,\dots,k_r \geq 1} g_{k_1,\dots,k_r} X_1^{k_1-1} \dots X_r^{k_r-1} \\
	&= \sum_{0 < n_1 < \dots < n_r} \frac{e^{X_1} q^{n_1}}{1-e^{X_1} q^{n_1}} \dots  \frac{e^{X_r} q^{n_r}}{1-e^{X_r} q^{n_r}} \,. 
	\end{split}
	\end{align}

\subsection{The series $ g^\sh_{k_1,\dots,k_r}$ and the map $\mg^\sh$}

Following \cite{BT} we define for $n_1,\dots,n_r \geq 1$ the following series
\[ \HH{n_1,\dots,n_r}{X_1,\dots,X_r} = \sum_{0 < d_1 < \dots < d_r} e^{d_1 X_1} \left( \frac{q^{d_1}}{1-q^{d_1}} \right)^{n_1} \dots  e^{d_r X_r} \left( \frac{q^{d_r}}{1-q^{d_r}} \right)^{n_r} \,.  \] 
Observe that by \eqref{eq:genseriesg}, \eqref{eq:partition} and \eqref{eq:genserg2} we have
\begin{equation}\label{eq:TandH}	\geng(X_1,\dots,X_r) = \HH{1,\dots,1}{X_r-X_{r-1}, X_{r-1}-X_{r-2},\dots,X_1}\,. \end{equation}	

%
\begin{definition} \begin{enumerate}[i)]
\item For $k_1,\dots,k_r \geq 1$ define the $q$-series $g^\sh_{k_1,\dots,k_r}(q) \in \Q[[q]]$ as the coefficients of the following generating function:
	\begin{align*}
\geng_\sh(X_1,\dots&,X_r) =\sum_{k_1,\dots,k_r \geq 1} g^\sh_{k_1,\dots,k_r}(q) X_1^{k_1-1} \dots X_r^{k_r-1}  \\
&:=  \sum_{m=1}^{r} \sum_{\substack{i_1+\dots+i_m = r\\i_1,\dots,i_m \geq 1}} \frac{1}{i_1! \dots i_m!} \HH{i_1,i_2, \dots,i_m}{X_r-X_{r-i_1},X_{r-i_1}-X_{r-i_1-i_2}, \dots, X_{i_m}}\,. 
	\end{align*}
Again we also write $g^\sh_{k_1,\dots,k_r}$ instead of $g^\sh_{k_1,\dots,k_r}(q)$.	
\item Define the $\Q$-linear map $\mg^\sh $ from $\h^1$ to $\Q[[q]]$ on the monomials by
\begin{align*}
\mg^\sh : \h^1 &\longrightarrow \Q[[q]] \,,\\
w = e_{k_1} \dots e_{k_r} &\longmapsto \mg^\sh(w) := g^{\sh}_{k_1,\dots,k_r} \,,
\end{align*}
set $\mg^\sh(1) = 1$ and extend it linearly to $\h^1$.
\end{enumerate}
\end{definition}

\begin{theorem}\label{thm:gshfacts}
	\begin{enumerate}[i)]
		\item For all $k_1,\dots,k_r \geq 1$ we have $g^{\sh}_{k_1,\dots,k_r} \in \G$.
		\item In the cases $k_1,\dots,k_{r_1} \geq 2$, $k_r \geq1$ it is  $g^{\sh}_{k_1,\dots,k_r} = g_{k_1,\dots,k_r} \in \G^{(0)}$.
		\item The map $\mg^\sh$ is an algebra homomorphism from $\h_\sh^1$ to $\G$.
	\end{enumerate}
\end{theorem}
\begin{proof} This is Proposition 5.5 together with Theorem 5.7 in \cite{B}, where the series $g^{\sh}_{k_1,\dots,k_r}$ is denoted $[k_r,\dots,k_1]^\sh$. Statement iii) was originally proven in \cite{BT}, where also a slightly weaker version of ii) can be found. Since we will need some parts of the proof we will recall the basic ideas:
	\begin{enumerate}[i)]
		\item To show that $g^{\sh}_{k_1,\dots,k_r} \in \G$ it is sufficient to prove that the coefficients of the series $H$ are elements in $\G$. This can be done by observing that $\left( \frac{e^X q^n}{1-e^X q^n} \right)^2= \frac{d}{dX}\frac{e^X q^n}{1-e^X q^n} - \frac{e^X q^n}{1-e^X q^n}$. Inductively this enables one to write the terms $\left( \frac{e^X q^n}{1-e^X q^n} \right)^n$, appearing in the definition of $H$, as derivatives of $\frac{e^X q^n}{1-e^X q^n}$, i.e. to write $H$ in terms of derivatives of $\geng$. Since the coefficients of $\geng$ are by definition in $\G$ the statement follows.
		\item To show that $g^{\sh}_{k_1,\dots,k_r} = g_{k_1,\dots,k_r} \in \G^{(0)}$ in the cases $k_1,\dots,k_{r_1} \geq 2$, $k_r \geq1$, one observes that there is just one summand in the definition of $g^{\sh}_{k_1,\dots,k_r}$, namely the case $i_1=\dots=i_m=1$, where all variables $X_1,\dots,X_{r-1}$ appear.  By \eqref{eq:TandH} this gives exactly $ g_{k_1,\dots,k_r} $.
		\item The statement that $\mg^\sh$ is an algebra homomorphism is equivalent to prove certain functional equations for the series $\geng_\sh$. This can be done by using results of Hoffman on quasi-shuffle product. In the lowest depth case this functional equation reads 
		\begin{equation}\label{eq:tshfcteq}
		\geng_\sh(X) \cdot 	\geng_\sh(Y) = 	\geng_\sh(X,X+Y) + 	\geng_\sh(Y,X+Y)\,,
		\end{equation}	
		which we will use later in the proof of Theorem \ref{thm:len1}.
	\end{enumerate}
\end{proof}
Due to the proof of Theorem \ref{thm:gshfacts} i) , writing $g^\sh$ as elements in $\G$ can be done explicitly:
\begin{proposition} \label{prop:gshing}
	\begin{enumerate}[i)]
		\item In depth two it is
		\[ g^{\sh}_{k_1,k_2}  = g_{k_1,k_2} + \delta_{k_1,1} \cdot \frac{1}{2} \left( g^{(1)}_{k_2} -  g_{k_2} \right) \, \]
		\item And in depth three it is
		\begin{align*}
		g^{\sh}_{k_1,k_2,k_3} &=  g_{k_1,k_2,k_3} + \delta_{k_1,1} \cdot \frac{1}{2} \left(   g^{(1,0)}_{k_2,k_3} -   g_{k_2,k_3}  \right)  +\delta_{k_2,1} \cdot \frac{1}{2} \left( g^{(0,1)}_{k_1,k_3} - g^{(1,0)}_{k_1,k_3} -  g_{k_1,k_3}   \right)  \\
		&+\delta_{k_1\cdot k_2,1}\cdot \left( \frac{1}{6} g^{(2)}_{k_3} - \frac{1}{4}g^{(1)}_{k_3}  +  \frac{1}{6} g_{k_3}   \right) \,.
		\end{align*}
	\end{enumerate}
	Here $\delta_{a,b}$ denotes the Kronecker delta which is $1$ in the case $a=b$ and $0$ otherwise.
\end{proposition}
\begin{proof}
This is i) and ii) of Corollary 5.8 in \cite{B}. 
\end{proof}

\section{Derivatives}

In this section we will discuss the behavior of the above introduced $q$-series under the operator $\dif = q \frac{d}{dq}$. Since this operator acts on a $q$-series by $\dif \sum_{n\geq 0} a_n q^n = \sum_{n>0} n a_n q^n$ it is easy to see that by definition we have
	\begin{equation} \label{eq:difong}  \dif g^{(d_1,\dots,d_r)}_{k_1,\dots,k_r}  = \sum_{j=1}^{r} (d_j+1) \cdot k_j \cdot g^{(d_1,\dots,d_j+1,\dots,d_r)}_{k_1,\dots,k_j+1,\dots,k_r} \,. 
		\end{equation}
In particular it follows that the space $\G$ is closed under $\dif$.

\subsection{Derivatives of $g$ and $g^\sh$}\label{sec:difg}
 In \cite{BK} it was proven, that also the subspace $\G^{(0)}$ is closed under the operator $\dif$ (Theorem 1.7 \cite{BK}). This is not obvious at all, since by \eqref{eq:difong}   we have for example
 	\[ \dif  g_{k_1,k_2} = \dif  g^{(0,0)}_{k_1,k_2} = k_1 g^{(1,0)}_{k_1+1,k_2} + k_2 g^{(0,1)}_{k_1,k_2+1} \,,\]
 	and a priori $g^{(1,0)}_{k_1+1,k_2}$ and $g^{(0,1)}_{k_1,k_2+1} $  are not elements in $\G^{(0)}$. 
In \cite{BK} the authors also give explicit formulas for $\dif g_k$ and $\dif g_{k_1,k_2}$. Numerical experiments suggest, that also the space spanned by all $g^\sh$ is closed under $\dif$, but so far there are no known results on this.  We now give the first result on this observation by the following explicit formula for $\dif g_k^\sh$. 
\begin{theorem}(Theorem \ref{thm:intro} i)) \label{thm:len1} For $k\geq 1$ and $\dif = q \frac{d}{dq}$ we have
	\begin{equation} \label{eq:thm1} 
	\frac{1}{k} \dif g^\sh_k = (k+1) g^\sh_{k+2} - \sum_{n=2}^{k+1} (2^n - 2) g^\sh_{k+2-n,n} \,.
	\end{equation}
\end{theorem}
\begin{proof}

To prove \eqref{eq:thm1} we will construct the generating functions of both sides and then show that they are equal. First notice that
\begin{align}\label{eq13}
\begin{split}
\dif \geng_\sh(Y) = q \frac{d}{dq} \geng_\sh(Y) &=  q \frac{d}{dq} \HH{1}{Y} = \sum_{0<d} e^{dY} q \frac{d}{dq} \left(\frac{q^d}{1- q^d}  \right)  \\
&= \sum_{0<d} d e^{dY} \left(  \left(\frac{q^d}{1-q^d}\right)^2+\frac{q^d}{1-q^d} \right)= \frac{d}{dY}\left( \HH{2}{Y}+\HH{1}{Y} \right) \,.\end{split}
\end{align}
Applying $\int_0^X \dots dY$ to both sides of \eqref{eq13} and using $\HH{2}{0}+\HH{1}{0}=g_2 $ we obtain
\begin{align*}
g_2 + \sum_{k>0} \frac{1}{k} \dif g_k X^k
= \HH{1}{X} + \HH{2}{X} = \geng(X) + \HH{2}{X} \,.
\end{align*}
This is the generating series of the left-hand side of \eqref{eq:thm1}, where we also included the term $g_2$ in the case $k=0$. This will also be included in the generating function of $(k+1)g^\sh_{k+2}$ for which we get 
\[ \sum_{k\geq0} (k+1) g^\sh_{k+2} X^k = \sum_{k>1} (k-1) g^\sh_k X^{k-2} = \frac{d}{dX} \sum_{k>0} g^\sh_k X^{k-1} = \frac{d}{dX} \geng_\sh(X)= \frac{d}{dX} \geng(X) \,.\] 
The generating function of the second part on the right-hand side of \eqref{eq:thm1} is given by 
\[ \sum_{k>0} \left( \sum_{n=2}^{k-1} (2^n-2) g^\sh_{k+2-n,n}\right) X^k = 2 \geng_\sh(X,2X) - 2 \geng_\sh(X,X) \,.\]

We therefore need to show
	\begin{equation}\label{eq:maineq} \geng(X) + \HH{2}{X} \overset{!}{=} \frac{d}{dX} \geng(X) - 2 \geng_\sh(X,2X) + 2 \geng_\sh(X,X) \,.	
	\end{equation}

	Using the shuffle product formula \eqref{eq:tshfcteq} for $\geng_\sh$, we obtain
	\begin{equation}\label{eq:gsquare1}
	\geng(X)^2 =\geng_\sh(X)^2 = \geng_\sh(X,X+X) + \geng_\sh(X,X+X) = 2 \geng_\sh(X,2X) \,.
	\end{equation}
	Using $\left( \frac{e^X q^n}{1-e^X q^n} \right)^2= \frac{d}{dX}\frac{e^X q^n}{1-e^X q^n} - \frac{e^X q^n}{1-e^X q^n}$ we also derive
	\begin{equation}\label{eq:gsquare2}
	\geng(X)^2 = 2 \geng(X,X) + \frac{d}{dX} \geng(X) - \geng(X)\,.
	\end{equation}
	Combining \eqref{eq:gsquare1} and \eqref{eq:gsquare2}  we obtain
	\begin{equation}\label{eq:gshX2X}
	2 \geng_\sh(X,2X) = 2 \geng(X,X) + \frac{d}{dX} \geng(X) - \geng(X) \,.
	\end{equation}
	By definition of $\geng_\sh$ we have
	\begin{equation}\label{eq:gshXX}
	2 \geng_\sh(X,X) = 2 \geng(X,X) + \HH{2}{X} \,.
	\end{equation}
	Equation \eqref{eq:maineq} now follows by combining \eqref{eq:gshX2X} and \eqref{eq:gshXX}\,.
\end{proof}

\begin{remark}
	Multiplying both sides in Theorem \ref{thm:len1} with $(1-q)^{k-2}$, taking the limit $q \rightarrow 1$ and making a shift from $k+2$ to $k$ we get as a Corollary for $k \geq 3$ the following formula
	\[ (k-1) \zeta(k) = \sum_{n=2}^{k-1} (2^n-2) \zeta(k-n,n) \,,\]
	which is a combination of the classical and the weighted sum formula (\cite{OZ}, Theorem 3) for double zeta values.
\end{remark}

\subsection{Multiple Eisenstein series and derivatives of $g^\sh$}

As mentioned in the introduction our motivation of studying the series $g^\sh$ are their appearance in the Fourier expansion of the multiple Eisenstein series $G^\sh_{k_1,\dots,k_r} \in \C[[q]]$. For the $\Q$-vector space spanned by all multiple Eisenstein series of weight $k$ we write 
\[ \mes_{k} := \big\langle G^\sh_{k_1,\dots,k_r}  \in \C[[q]] \mid k_1+\dots+k_r = k\,, 0\leq r \leq k \big\rangle_{\Q} \quad\text{and set}\quad \mes= \sum_{k\geq 0} \mes_k\,. \]

The connection of $G^\sh$ and $g^\sh$ is given by a complicated but explicit formula, the Goncharov coproduct, in \cite{BT}. By abuse of notation we will consider $G^\sh$ as a $\Q$-linear map  

\begin{align*}
G^\sh : \h^1 &\longrightarrow \C[[q]] \,,\\
w = e_{k_1} \dots e_{k_r} &\longmapsto G^\sh(w) := G^\sh_{k_1,\dots,k_r} \,.
\end{align*}
As shown in \cite{BT}, this map is an algebra homomorphism from $\h_\sh^1$ to $\C[[q]]$.
 Recall that we defined for words $u,v \in \h^1$ the element $\ds(u,v) \in \h^1$ by
 \[ \ds(u,v) = u \ast v - u \shh v \,. \]
 As seen in Theorem \ref{thm:edsh} it is $\zeta^\sh(\ds(u,v) ) =0$ for all $u \in \h^1$, $v\in \h^0$ and conjecturally these give all relations between multiple zeta values. A natural question therefore is, in which cases we have  $G^\sh(\ds(u,v) ) =0$. This will not be the case for all $u \in \h^1$ and $v\in \h^0$ and we will see below, that the failure of the extended double shuffle relations for multiple Eisenstein series has a connection to the action of the operator $\dif$. But since the definition of $G^\sh$ is quite complicated, we need to restrict our attention to the series $g^\sh$. Luckily numerical calculations suggests, that these two objects have a really close connection.  
 To make clear what we mean by this we first define for $k\geq 1$
\[ \G^\sh_{\leq k} := \big\langle g^\sh_{k_1,\dots,k_r}  \in \G \mid k_1+\dots+k_r \leq k\,, 0\leq r \leq k \big\rangle_{\Q} \,. \]

The motivation for this are the following questions, which are all motivated by numerical experiments and which are expected to be true.
\begin{question}\label{qu:mes}
	\begin{enumerate}[i)] 
		\item Do we have $(2\pi i)^2 \dif \mes_k \subset \mes_{k+2}$ and $\dif \G^\sh_{\leq k} \subset \G^\sh_{\leq k+2}$? 
		\item Is the map $F$, given by
		\begin{align*}
		F: \mathcal{E}_k &\longrightarrow \bigslant{\G^\sh_{\leq k} }{\G^\sh_{\leq k-1} } \\
		G^{\sh}_{k_1,\ldots,k_r} &\longmapsto 	g^{\sh}_{k_1,\ldots,k_r}\,,
		\end{align*}
		an isomorphism of $\Q$-vector spaces? 
		\item Assuming i), does the map $F$ satisfy $d F(f) = F((2\pi i)^2 \dif f)$ for all $f \in \mes_k$?
	\end{enumerate}
\end{question}

	\begin{proposition}\label{prop:dgk}
		For $k\geq 1$ we have 
		\begin{align*}
		\dif  g^\sh_{k}  &\equiv 2 k \cdot  \mg^\sh(\ds(e_1,e_{k+1}))  \qquad  \qquad \,\mod \G^\sh_{\leq k+1}
		\end{align*} 
	\end{proposition}
\begin{proof}	
Notice that by Proposition \ref{prop:gshing}
	\begin{align*}
	g^{\sh}_{k_1,k_2}  = g_{k_1,k_2} + \delta_{k_1,1} \cdot \frac{1}{2}  g^{(1)}_{k_2}  \,\mod \G^\sh_{\leq k_1+k_2-1}
	\end{align*} 
	and since the quasi-shuffle product $\boxast$ equals the harmonic product $\ast$ if one divides out lower weight, it is
	\begin{align*}
	g^{\sh}_{k_1} \cdot g^{\sh}_{k_2} =g_{k_1} \cdot g_{k_2} \equiv g_{k_1,k_2}+g_{k_2,k_1}+g_{k_1+k_2}    \,\mod \G^\sh_{\leq k_1+k_2-1}\,.
	\end{align*} 
	With this we obtain 
	\begin{align*}
	\mg^\sh(\ds(e_{k_1},e_{k_2})) &= \mg^\sh(e_{k_1} \ast e_{k_2}) -  \mg^\sh(e_{k_1} \shh e_{k_2})\\
	&= g^\sh_{k_1,k_2} +  g^\sh_{k_2,k_1} +  g^\sh_{k_1+k_2} - g^{\sh}_{k_1} \cdot g^{\sh}_{k_2}  \\
	&\equiv \frac{1}{2} \delta_{k_1,1} g^{(1)}_{k_2} + \frac{1}{2} \delta_{k_2,1} g^{(1)}_{k_1}  \qquad  \qquad \,\mod \G^\sh_{\leq k_1+k_2-1}\,.
	\end{align*} 
	The statement now follows since $\dif  g^\sh_{k} = k \cdot  g^{(1)}_{k+1}$.
\end{proof}
\begin{remark}
We remark that the explicit expression of $\dif g^\sh_k$ in Theorem \ref{eq:thm1} can also be written as
\[	\frac{1}{k} \dif  g^\sh_{k} = \sum_{i=1}^{k-1} \mg^\sh(\ds(e_i,e_{k+2-i})) \,.\]
Therefore from Proposition \ref{prop:dgk} we can deduce $\sum_{i=2}^{k-2} \mg^\sh(\ds(e_i,e_{k+2-i})) \in \G^\sh_{\leq k+1}$.
\end{remark}
Considering question \ref{qu:mes} one should have the same formula for the derivative of Eisenstein series as the above Proposition. This is indeed the case:
\begin{theorem}
	For $k\geq 1$, the derivative of the Eisenstein series $G^\sh_k$ is given by
	\begin{align*}
	(2\pi i)^2 \dif G^\sh_{k} &= 2 k \cdot G^\sh(\ds(e_1,e_{k+1}))= G^\sh_{1,k+1} + G^\sh_{k+1,1} + G^\sh_{k+2} - G^\sh_{k+1} \cdot G^\sh_{1} \in \mes_{k+2}\,.
	\end{align*}
\end{theorem}
\begin{proof}
This was first proven by M. Kaneko in an unpublished work. It can also be obtained by using the explicit formulas for the Fourier expansions of Double Eisenstein series presented in \cite{BT} and the quasi-shuffle product formula for the functions $g$ introduced in the beginning.
\end{proof}

We now want to give the depth $2$ and $3$ version of Proposition \ref{prop:dgk}. For this we need the following two Lemma.

\begin{lemma}\label{lem:g10ingsh} For $k_1,k_2 \geq 2$ it is $g^{(1,0)}_{k_1,k_2} , g^{(0,1)}_{k_1,k_2} \in \G^\sh_{\leq k_1+k_2+1}$.
\end{lemma}
\begin{proof} 
	Recall that we have  $g_k = g^\sh_k$  for all $k\geq 1$ and $g_{a,b} = g^\sh_{a,b}$ when $a>1$ and $b\geq 1$.
	First we notice that also $g_{1,b} \in \G^\sh_{\leq b+1}$ for all $b \geq 1$: By the quasi-shuffle product it is $g_1 \cdot g_b = g_{1,b} + g_{b,1}+ \sum_{j=1}^{b+1} \alpha_j g_j$ for some $\alpha_j \in \Q$. Since $g_1 \cdot g_b,\, g_j \in \G^\sh_{b+1}$ we deduce  $g_{1,b} \in \G^\sh_{\leq b+1}$.
	
	Now consider the quasi-shuffle product in depth $3$
	\[ g_1 \cdot g_{k_1,k_2}  = g_{1,k_1,k_2} + g_{k_1,1,k_2} + g_{k_1,k_2,1} + \sum_{a+b \leq k_1+k_2} \beta_{a,b} \cdot g_{a,b} \,, \]
	for some $\beta_{a,b} \in \Q$. Since for $k_1,k_2 \geq 2$  we have $ g_{a,b},\,g_1 \cdot g_{k_1,k_2}, \, g_{k_1,k_2,1}  \in \G^\sh_{\leq k_1+k_2+1}$ it follows that  $g_{1,k_1,k_2} + g_{k_1,1,k_2} \in \G^\sh_{\leq k_1+k_2+1}$. 
	
	Using the explicit formula for $g^\sh_{a,b,c}$ from Proposition \ref{prop:gshing} it is easy to see that for $k_1,k_2 \geq 2$ 
	
	\[2 g^\sh_{1,k_1,k_2} + 2 g^\sh_{k_1,1,k_2} = 2 g_{1,k_1,k_2} + 2 g_{k_1,1,k_2} - 2g_{k_1,k_2} + g^{(0,1)}_{k_1,k_2} \,.\]
From this we observe $g^{(0,1)}_{k_1,k_2} \in \G^\sh_{\leq k_1+k_2+1}$ since by the discussion above every other term in this equation is also in $\G^\sh_{\leq k_1+k_2+1}$.
	
		Now we want to show that also $g^{(1,0)}_{k_1,k_2} \in \G^\sh_{\leq k_1+k_2+1}$. For this consider again that for some $\gamma_j \in \Q$ the quasi-shuffle product of $g^{(1)}_{k_1} \cdot g_{k_2}$ reads
\[  g^{(1)}_{k_1} \cdot g_{k_2} = g^{(1,0)}_{k_1,k_2} + g^{(0,1)}_{k_2,k_1}+ \sum_{j=1}^{k_1+k_2} \gamma_j g^{(1)}_j \,.\]

By Theorem \ref{thm:len1} we know that $g^{(1)}_j = \frac{1}{(j-1)} \dif g_{j-1}$ is again an Element in $\G^\sh_{\leq k_1+k_2+1}$ for $j \leq k_1+k_2$. Since we proved $g^{(0,1)}_{k_1,k_2} \in \G^\sh_{\leq k_1+k_2+1}$ above we therefore also obtain that  $g^{(1,0)}_{k_1,k_2} \in \G^\sh_{\leq k_1+k_2+1}$.
	
\end{proof}
%
%

Similar to the depth $1$ case we will "measure" the failure of the double shuffle relations of $g^\sh$ and then relate this to the action of the operator $\dif$.
\begin{lemma}\label{lem:gdsh1} Let $k_1,k_2,k_3,k_4 \geq 1$ and $k=k_1+\dots+k_4$ be such that there is exactly one index $1 \leq j \leq 4$ with $k_j=1$. Then  we have
\begin{enumerate}[i)]	
	\item $\begin{aligned}[t]
		\mg^\sh(\ds(e_{k_1}\,,\,e_{k_2}e_{k_3}))\equiv &\,\,\, \delta_{k_1,1} \frac{1}{2} g^{(0,1)}_{k_2,k_3} +  \delta_{k_3,1} \frac{1}{2}\left( g^{(0,1)}_{k_2,k_1} - g^{(1,0)}_{k_2,k_1}\right)   \mod \G^\sh_{\leq k_1+k_2+k_3-1} \,.
	\end{aligned}$
		
	\item $\begin{aligned}[t]
	\mg^\sh(\ds(e_{k_1}\,,\,e_{k_2}e_{k_3}e_{k_4}))\equiv &\,\,\, \delta_{k_1,1} \frac{1}{2} g^{(0,0,1)}_{k_2,k_3,k_4} +  \delta_{k_4,1} \frac{1}{2}\left( g^{(0,0,1)}_{k_2,k_3,k_1} - g^{(0,1,0)}_{k_2,k_3,k_1}\right) \mod \G^\sh_{\leq k-1} \,.
	\end{aligned}$
	
	\item $\begin{aligned}[t]
		\mg^\sh(\ds(e_{k_1}e_{k_2}\, , \,e_{k_3}e_{k_4}))\equiv& \,\,\, \delta_{k_2,1} \frac{1}{2} \left( g^{(0,0,1)}_{k_1,k_3,k_4} - g^{(1,0,0)}_{k_1,k_3,k_4} + g^{(0,0,1)}_{k_3,k_1,k_4} - g^{(0,1,0)}_{k_3,k_1,k_4}   \right)  \\
		&+ \delta_{k_2,1} \frac{1}{2} \left(g^{(0,1)}_{k_1+k_3,k_4}-g^{(1,0)}_{k_1+k_3,k_4}          \right)  \\
		&+\delta_{k_4,1} \frac{1}{2} \left( g^{(0,0,1)}_{k_1,k_3,k_2} - g^{(0,1,0)}_{k_1,k_3,k_2} + g^{(0,0,1)}_{k_3,k_1,k_2} - g^{(1,0,0)}_{k_3,k_1,k_2}   \right)  \\
		&+ \delta_{k_4,1} \frac{1}{2} \left(g^{(0,1)}_{k_1+k_3,k_2}-g^{(1,0)}_{k_1+k_3,k_2}\right)   \,\,\,\,\,\, \qquad \qquad \mod \G^\sh_{\leq k-1} \,.
	\end{aligned}$
\end{enumerate}		
\end{lemma}

\begin{proof}
	
		i)  Since $g_{k_1,k_3} \in \G^\sh_{\leq k_1+k_2+k_3-1}$ and we assume that there is just one index $j$ with $k_j=1$, i.e. the term with $\delta_{k_1 \cdot k_2,1}$ does not play a role, we get by Proposition \ref{prop:gshing} that 
		\begin{align}\label{eq:gsh3mod}
		g^{\sh}_{k_1,k_2,k_3} &\equiv g_{k_1,k_2,k_3} + \delta_{k_1,1} \cdot \frac{1}{2}    g^{(1,0)}_{k_2,k_3}   +\delta_{k_2,1} \cdot \frac{1}{2} \left( g^{(0,1)}_{k_1,k_3} - g^{(1,0)}_{k_1,k_3}   \right) \mod \G^\sh_{\leq k_1+k_2+k_3-1}
		\end{align}
and therefore
\begin{align*}
\mg^\sh(e_{k_1} \ast e_{k_2} e_{k_3}) &=g^\sh_{k_1,k_2,k_3}+g^\sh_{k_2,k_1,k_3}+g^\sh_{k_2,k_3,k_1}+g^\sh_{k_1+k_2,k_3}+g^\sh_{k_2,k_1+k_3}\\
&\equiv \mg(e_{k_1} \ast e_{k_2} e_{k_3}) +\delta_{k_1,1} \cdot \frac{1}{2}    g^{(1,0)}_{k_2,k_3}   +\delta_{k_2,1} \cdot \frac{1}{2} \left( g^{(0,1)}_{k_1,k_3} - g^{(1,0)}_{k_1,k_3}   \right)  \\
&+\delta_{k_2,1} \cdot \frac{1}{2}    g^{(1,0)}_{k_1,k_3}   +\delta_{k_1,1} \cdot \frac{1}{2} \left( g^{(0,1)}_{k_2,k_3} - g^{(1,0)}_{k_2,k_3}   \right) \\
&+\delta_{k_2,1} \cdot \frac{1}{2}    g^{(1,0)}_{k_3,k_1}   +\delta_{k_3,1} \cdot \frac{1}{2} \left( g^{(0,1)}_{k_2,k_1} - g^{(1,0)}_{k_2,k_1}   \right) \\
&+ \delta_{k_2,1} \frac{1}{2}g^{(1)}_{k_1+k_3} \qquad \qquad\qquad\qquad\qquad\qquad\qquad\mod \G^\sh_{\leq k_1+k_2+k_3-1}\,.
\end{align*}		
On the other hand we have
\begin{align*}
\mg^\sh(e_{k_1} \shh e_{k_2} e_{k_3}) &=\mg^\sh(e_{k_1}) \cdot  \mg^\sh(e_{k_2} e_{k_3})\\
&\equiv g_{k_1} \cdot \left( g_{k_2,k_3} + \delta_{k_2,1} \frac{1}{2}\left(  g^{(1)}_{k_3}  - g_{k_3}\right)\right)\\
&\equiv  \mg(e_{k_1} \ast e_{k_2} e_{k_3}) + \delta_{k_2,1} \frac{1}{2} \left( g^{(0,1)}_{k_1,k_3}+g^{(1,0)}_{k_3,k_1}+g^{(1)}_{k_1+k_3}\right) \mod \G^\sh_{\leq k_1+k_2+k_3-1}
\end{align*}
Here we used again that the extra terms appearing in the quasi-shuffle product all vanish since they are elements in $\G^\sh_{\leq k_1+k_2+k_3-1}$. For the product $g_{k_1} \cdot g_{k_3}^{(1)}$ this is the case because we know  by Theorem \ref{thm:len1} that $g^{(1)}_j \in \G^\sh_{\leq k_1+k_2+k_3-1}$ for $j<k_1+k_2+k_3-1$.

The result follows from $\mg^\sh(\ds(e_{k_1},e_{k_2}e_{k_3})) = \mg^\sh(e_{k_1} \ast e_{k_2} e_{k_3})-\mg^\sh(e_{k_1} \shh e_{k_2} e_{k_3})$.	\\

To prove ii) and iii) we use the same idea as in i). First calculate  $\mg^\sh(e_{k_1}\ast e_{k_2}e_{k_3}e_{k_4})$ and $\mg^\sh(e_{k_1}\ast e_{k_2}e_{k_3}e_{k_4})$ by using \eqref{eq:gsh3mod} and the following formula, which can be obtained using the same technique as in the proof of Proposition  \ref{prop:gshing} together with our assumptions on the $k_j$:
\begin{align*}
		g^{\sh}_{k_1,k_2,k_3,k_4} &\equiv g_{k_1,k_2,k_3,k_4} + \delta_{k_1,1} \cdot \frac{1}{2}    g^{(1,0,0)}_{k_2,k_3,k_4}   +\delta_{k_2,1} \cdot \frac{1}{2} \left( g^{(0,1,0)}_{k_1,k_3,k_4} - g^{(1,0,0)}_{k_1,k_3,k_4}   \right)\\
		&+\delta_{k_3,1} \frac{1}{2} \left( g^{(0,0,1)}_{k_1,k_2,k_4} - g^{(0,1,0)}_{k_1,k_2,k_4}   \right) \mod \G^\sh_{\leq k-1}\,.
\end{align*}
When calculating $\mg^\sh(e_{k_1}\shh e_{k_2}e_{k_3}e_{k_4})$ and $\mg^\sh(e_{k_1}\shh e_{k_2}e_{k_3}e_{k_4})$ one derives again the the quasi-shuffle products  and then apply Lemma \ref{lem:g10ingsh} to argue why the appearing error terms of the form $g^{(1,0)}_{a,b}$ and $g^{(0,1)}_{a,b}$ with $a,b \geq 2$ vanish. 
\end{proof}

\begin{remark}
Since it is expected that $\mes_k$ and $\bigslant{\G^\sh_{\leq k} }{\G^\sh_{\leq k-1} }$ are isomorphic as $\Q$-vector spaces, Lemma \ref{lem:gdsh1} can be used to guess which extended double shuffle relations are fulfilled by multiple Eisenstein series. In \cite{BT} it is proven, that  for $k_1,k_2,k_3 \geq 2$ 
\begin{equation}\label{eq:dshmes}
G^\sh(\ds(e_{k_1}\,,\,e_{k_2}e_{k_3}))=0 \,.
\end{equation}
But due to Lemma \ref{lem:gdsh1}  i)
it is expected that \eqref{eq:dshmes} also holds for the cases $k_2=1$ and $k_1,k_3 \geq 2$. In other words the triple Eisenstein series may satisfy all finite double shuffle relations. The special case $k_2=1$ and $k_1=k_3=2$ of \eqref{eq:dshmes} was proven in \cite{B} Example 6.14.
\end{remark}

\begin{theorem}\label{thm:dgshlen2len3}
For $k_1, k_2, k_3\geq 2$ and $\dif = q \frac{d}{dq}$ we have
\begin{enumerate}[i)]
	\item $\begin{aligned}[t]
	\dif  g^\sh_{k_1,k_2} &\equiv 2 k_1 \left( \mg^\sh(\ds(e_1,e_{k_1+1} e_{k_2})) -  \mg^\sh(\ds(e_{k_2},e_{k_1+1} e_{1}))  \right) \\ 
	&+ 2 k_2 \cdot \mg^\sh (\ds(e_1,e_{k_1} e_{k_2+1}))  \qquad \qquad \qquad \qquad \qquad  \qquad \,\mod \G^\sh_{\leq k_1+k_2+1}
	\end{aligned}$
	\item  $\begin{aligned}[t]
	\dif  g^\sh_{k_1,k_2,k_3} \equiv \,\,  &2k_1 \cdot\mg^\sh( \ds(e_1,e_{k_1+1} e_{k_2} e_{k_3}) + \ds(e_{k_3},e_{k_2} e_{k_1+1} e_{1}) )\\ 
	+&2k_1 \cdot\mg^\sh( \ds(e_{k_3}, e_{k_1+1+k_2} e_{1}) - \ds(e_{k_1+1} e_1, e_{k_2} e_{k_3})  ) \\
	+&2 k_2  \cdot\mg^\sh(\ds(e_1,e_{k_1} e_{k_2+1} e_{k_3}) - \ds(e_{k_3},e_{k_1} e_{k_2+1} e_{1})  ) \\ 
	+ &2 k_3 \cdot \mg^\sh (\ds(e_1,e_{k_1} e_{k_2} e_{k_3+1}))  \qquad \qquad \qquad \qquad  \qquad \mod \G^\sh_{\leq k_1+k_2+k_3+1}
	\end{aligned}$
\end{enumerate}
\end{theorem}

\begin{proof} i) Since for  $k_1, k_2\geq 2$ it is $g^\sh_{k_1,k_3} = g_{k_1,k_2} = g^{(0,0)}_{k_1,k_2}$ we have by \eqref{eq:difong} that
\[ \dif  g^\sh_{k_1,k_2} = k_1 g^{(1,0)}_{k_1+1,k_2} + k_2 g^{(0,1)}_{k_1,k_2+1} \,.\]
By Lemma \ref{lem:gdsh1} we obtain 
\begin{align*}
\frac{1}{2}g^{(1,0)}_{k_1+1,k_2} &\equiv  \mg^\sh(\ds(e_1,e_{k_1+1} e_{k_2})) -  \mg^\sh(\ds(e_{k_2},e_{k_1+1} e_{1})   \mod \G^\sh_{\leq k_1+k_2+1}\,, \\
\frac{1}{2}g^{(0,1)}_{k_1,k_2+1} &\equiv \mg^\sh (\ds(e_1,e_{k_1} e_{k_2+1}) \mod \G^\sh_{\leq k_1+k_2+1} \,,
\end{align*}
from which the statement follows.

ii) Similar to i) one uses Lemma \ref{lem:gdsh1} to get explicit formulas for  $g^{(1,0,0)}_{k_1+1,k_2,k_3}$, $g^{(0,1,0)}_{k_1,k_2+1,k_3}$ and $g^{(0,0,1)}_{k_1,k_2,k_3+1}$, which we will omit here since the calculation is easy but messy.
\end{proof}

From Theorem \ref{thm:dgshlen2len3} the statement of Theorem \ref{thm:intro} ii) follows.

\begin{ex}
\begin{align*}
\dif  g^\sh_{2,2} \equiv \,\,\,&4 g^\sh_{2, 4} + 4 g^\sh_{3, 3} + 4 g^\sh_{4, 2} - 4 g^\sh_{5, 1} \\&- 
4 g^\sh_{1, 2, 3} + 4 g^\sh_{1, 3, 2} + 24 g^\sh_{1, 4, 1} - 
4 g^\sh_{2, 1, 3} - 4 g^\sh_{2, 2, 2} + 8 g^\sh_{2, 3, 1} \, \mod\G^\sh_{\leq 5}
\end{align*}
\end{ex}

\begin{conjecture}
	 For $k_1,k_2,k_3 \geq 2$ the derivative of the Double and Triple Eisenstein series  are given by
	 \begin{align*}
	 (- 2\pi i)^2 \dif G^\sh_{k_1,k_2} =\,\, &2k_1 \left( G^\sh(\ds(e_1,e_{k_1+1} e_{k_2})) - G^\sh(\ds(e_{k_2},e_{k_1+1} e_{1}))  \right) \\ 
	 + &2 k_2 \cdot G^\sh (\ds(e_1,e_{k_1} e_{k_2+1})) \,,
	 \end{align*} and
	 \begin{align*}		
	 (- 2\pi i)^2 \dif G^\sh_{k_1,k_2,k_3} =\,\, &2k_1 \cdot G^\sh( \ds(e_1,e_{k_1+1} e_{k_2} e_{k_3}) + \ds(e_{k_3},e_{k_2} e_{k_1+1} e_{1}) )\\ 
	 +&2k_1 \cdot G^\sh( \ds(e_{k_3}, e_{k_1+1+k_2} e_{1}) - \ds(e_{k_1+1} e_1, e_{k_2} e_{k_3})  ) \\
	 +&2 k_2  \cdot G^\sh(\ds(e_1,e_{k_1} e_{k_2+1} e_{k_3}) - \ds(e_{k_3},e_{k_1} e_{k_2+1} e_{1})  ) \\ 
	 + &2 k_3 \cdot G^\sh (\ds(e_1,e_{k_1} e_{k_2} e_{k_3+1}))  \,.
	 \end{align*}
\end{conjecture}

%
%


\end{document}